\documentclass[bj]{imsart2}
\RequirePackage[OT1]{fontenc}
\RequirePackage{amsthm,amsmath}
\usepackage{natbib}
\usepackage{amsmath}
\usepackage{amssymb}
\usepackage{amsthm}
\usepackage{amsfonts}
\usepackage{latexsym}
\usepackage{sectsty}
\usepackage{epsfig}
\usepackage{url}
\usepackage{microtype}
\usepackage{fixmath}
\usepackage{subfigure}

\newcommand{\rmPi}{\mathrm{\Pi}}
\newcommand{\iid}{\stackrel{iid}{\sim}}
\newcommand{\ind}{\stackrel{ind}{\sim}}
\newcommand{\m}{\boldsymbol{\mathrm{m}}}
\newcommand{\Y}{\boldsymbol{\mathrm{Y}}}
\newcommand{\X}{\boldsymbol{\mathrm{X}}}

\newcommand{\nn}{\nonumber}
\newcommand{\rmGamma}{\mathrm{\Gamma}}

\newtheorem{theorem}{Theorem}
\newtheorem{lemma}{Lemma}

\newtheorem{corollary}{Corollary}
\newtheorem{definition}{Definition}

\def\<{\langle}
\def\>{\rangle}
\def\({\left(}
\def\){\right)}
\def\[{\left[}
\def\]{\right]}
\def\rm{\mathrm}

\def\E{\mathbb{E}}
\def\e{\mbox{e}}
\def\BP{\mbox{BP}}
\def\R{\mathbb{R}}
\def\I{\mathbb{I}}

\begin{document}

\begin{frontmatter}

\title{A Constructive Definition of the Beta Process}

\runtitle{A Constructive Definition of the Beta Process}

\begin{aug}
\author{\fnms{John} \snm{Paisley}\thanksref{a,e1}\ead[label=e1,mark]{jpaisley@columbia.edu}}
\and
\author{\fnms{Michael} \snm{I. Jordan}\thanksref{b,e2}\ead[label=e2,mark]{jordan@stat.berkeley.edu}}

\address[a]{Department of Electrical Engineering, Columbia University\\
\printead{e1}}

\address[b]{Department of Statistics and Department of EECS, University of California, Berkeley
\printead{e2}}

\runauthor{Paisley and Jordan}

\affiliation{Columbia University and University of California, Berkeley}

\end{aug}

\begin{abstract}
We derive a construction of the beta process that allows for the atoms with significant measure to be drawn first. Our representation is based on an extension of the \cite{Sethuraman:1994} construction of the Dirichlet process, and therefore we refer to it as a stick-breaking construction. Our first proof uses a limiting case argument of finite arrays. To this end, we present a finite sieve approximation to the beta process that parallels that of \cite{Ishwaran:2002a} and prove its convergence to the beta process. We give a second proof of the construction using Poisson process machinery. We use the Poisson process to derive almost sure truncation bounds for the construction. We conclude the paper by presenting an efficient sampling algorithm for beta-Bernoulli and beta-negative binomial process models.
\end{abstract}




\end{frontmatter}

\section{Introduction}

Stick-breaking constructions play an important role in Bayesian nonparametric models because they allow for the construction of infinite dimensional discrete measures with sparse weights \citep{Sethuraman:1994,Ishwaran:2001}. The canonical stick-breaking construction was derived by \cite{Sethuraman:1994} for the Dirichlet process \citep{Ferguson:1973}. According to this construction, given a probability measure $G_0$ and constant $\alpha > 0$, the random probability measure
\begin{equation}\label{eqn.DP}
 G = \sum_{i=1}^{\infty}V_i\prod_{j=1}^{i-1}(1-V_j)\delta_{\theta_i},\quad V_i\stackrel{iid}{\sim} \rm{Beta}(1,\alpha),\quad \theta_i \stackrel{iid}{\sim} G_0
\end{equation}
is a Dirichlet process, denoted $G \sim \mathrm{DP}(\alpha G_0)$. Dirichlet processes have a very developed literature in  Bayesian nonparametrics (see, e.g., \cite{BNP:2010}).

A recently growing research area within Bayesian nonparametrics has been on beta process priors and associated models. The beta process has been presented and developed in the statistics literature for applications in survival analysis \citep{Hjort:1990,Muliere:1997,Kim:2001,Lee:2004}. Recent developments in the statistical machine learning literature have found beta processes useful for nonparametric latent factor models \citep{Griffiths:2006,Thibaux:2007}; this perspective has found many applications \citep{Williamson:2010,Fox:2010,Paisley:2009,Zhou:2012b,Broderick:2012b}.

As with the Dirichlet process, representations are needed for working with the beta process. A marginalized representation for the beta-Bernoulli process called the Indian buffet process (IBP) \citep{Griffiths:2006} was recently presented that has the beta process as the underlying DeFinetti mixing measure \citep{Thibaux:2007}. Therefore, the IBP stands in similar relation to the beta process as the Chinese restaurant process does to the Dirichlet process \citep{Aldous:1985}. Recently, \cite{Paisley:2010} presented a method for explicitly constructing beta processes based on the notion of stick-breaking, and further developed these ideas in \cite{Paisley:2012}. In this paper we collect these results and also present new results on asymptotically correct finite approximations to the beta process, truncation bounds for working with the beta-negative Binomial process, and posterior sampling for the beta-Bernoulli and beta-negative binomial processes.

The paper is organized as follows: In Section \ref{sec.BP} we review the beta process and its connection to the Poisson process. We present a finite approximation to the beta process in Section \ref{sec.finite} and prove that it has the correct asymptotic distribution. We use this approximation in Section \ref{sec.stickbreaking} in our first proof of the stick-breaking construction of the beta process and present a second proof using the Poisson process machinery. In Section \ref{sec.truncations} we derive bounds on truncated stick-breaking approximations for models based on the beta-Bernoulli and beta-negative binomial processes. We then present simple posterior sampling schemes for these two processes in Section \ref{sec.MCMC}.

\section{Beta processes}\label{sec.BP}

Beta processes comprise a class of completely random measures \citep{Kingman:1967}. They are defined on an abstract measurable space $(\rm{\Theta},\mathcal{A})$, are almost surely discrete, and have the property that the mass of any particular atom lies in the interval $(0,1]$. The beta process has an underlying Poisson process driving it on the space $(\rm{\Theta} \times [0,1],\mathcal{A}\otimes\mathcal{B})$. We use this representation in the following definition.

\begin{definition}[Beta process]\label{def.bp1}
Let $N$ be a Poisson random measure on $(\rm{\Theta} \times [0,1],\mathcal{A}\otimes\mathcal{B})$ with mean measure $\nu(d\theta,d\pi) = \alpha(\theta)\pi^{-1}(1-\pi)^{\alpha(\theta)-1}d\pi\mu(d\theta)$, where $\mu$ is a diffuse $\sigma$-finite measure and the function $\alpha(\theta)$ is strictly positive and finite. For a compact set $A \in \mathcal{A}$, define the completely random measure
$$H(A) = \int_{A\times[0,1]} N(d\theta,d\pi)\pi.$$
Then $H$ is a beta process  with concentration $\alpha(\cdot)$ and base $\mu$, denoted $H \sim \BP(\alpha,\mu)$.
\end{definition}

The Poisson random measure $N(d\theta,d\pi)$ is a random counting measure such that for $S \in \mathcal{A}\otimes\mathcal{B}$, the distribution of $N(S)$ is Poisson with parameter $\nu(S)$. It is also completely random, since for any collection of pairwise disjoint sets $S_1,\dots,S_k \in \mathcal{A}\otimes\mathcal{B}$, the random variables $N(S_1),\dots,N(S_k)$ are independent, which by extension proves that $H$ is completely random \citep{Cinlar:2011}.

A fundamental aspect of the study of Poisson random measures is analyzing how they behave when integrated against a function---in this case studying the integral $\int N(d\theta,d\pi)f(\theta,\pi)$ over subsets of $\rm{\Theta}\times [0,1]$. Definition \ref{def.bp1} is the special case where $f(\theta,\pi) = \pi$ and the integral is taken over the entire interval $[0,1]$ in the $\pi$ dimension. It follows \citep{Cinlar:2011} that for $t < 0$ the Laplace functional of $H$ is
\begin{equation}\label{eqn.BP_laplace}
 \E\,\e^{tH(A)} = \exp\left\lbrace - \int_{A\times[0,1]}\nu(d\theta,d\pi)\(1-\e^{t\pi}\)\right\rbrace,
\end{equation}
with mean measure $\nu$ given in Definition \ref{def.bp1},
\begin{equation}\label{eqn.BP_mean_measure}
 \nu(d\theta,d\pi) = \alpha(\theta)\pi^{-1}(1-\pi)^{\alpha(\theta)-1}d\pi\mu(d\theta).
\end{equation}
The beta process is completely characterized by the mean measure of its underlying Poisson random measure. We divide this mean measure into two measures: A $\sigma$-finite and diffuse base measure $\mu$, and a transition probability kernel $\lambda(\theta,d\pi) = \alpha(\theta)\pi^{-1}(1-\pi)^{\alpha(\theta)-1}d\pi$, which is a measure on $[0,1]$ for each $\theta$ called the L\'{e}vy measure.

By defining the beta process as in Definition \ref{def.bp1} its existence immediately follows from the well-studied Poisson process machinery; since $\int_0^1 (|\pi|\wedge 1)\lambda(\theta,d\pi) = 1 < \infty$, by Campbell's theorem it follows that $H(A)$ is finite almost surely. The form of Equation (\ref{eqn.BP_laplace}) shows that $H$ is a pure-jump process, and the fact that $\lambda(\theta,[0,1]) = \infty$, but $\lambda(\theta,[\epsilon,1]) < \infty$ for all $\theta$ and $\epsilon > 0$ ensures that $H$ has an infinite number of jumps in any set $A\in\mathcal{A}$ for which $\mu(A)>0$, but only a finite number of magnitude greater than $\epsilon$.

Since $H$ is a pure-jump process with an infinite number of jumps, it can be written as a sum over delta measures; it will later be convenient to use two indices for this process. We can therefore write $H$ as either
\begin{equation}
 H = \sum_{i=1}^{\infty}\pi_i\delta_{\theta_i}, \quad \mbox{or equivalently} \quad H = \sum_{i=1}^{\infty}\sum_{j=1}^{C_i}\pi_{ij}\delta_{\theta_{ij}},
\end{equation}
where $C_i$ is a random variable that will appear later and is finite almost surely. We will also slightly abuse notation by letting $\theta\in H$ indicate an atom that has nonzero measure according to $H$.

As shown, a beta process is equivalently represented as a function of a Poisson random measure on the extended space $\rm{\Theta}\times [0,1]$. We indicate the set of atoms of this Poisson random measure by $\rmPi = \{(\theta,\pi)\}$, which is the set of locations with measure one according to $N$. The following two general lemmas about the resulting Poisson point process $\rmPi$ will be used later in this paper.

\begin{lemma}[Marked Poisson process (\cite{Cinlar:2011}, Theorem 6.3.2)]\label{lem.marked}
 Let $\rmPi^*$ be a Poisson process on $\rm{\Theta}$ with mean measure $\mu$. For each $\theta \in \rmPi^*$ associate a random variable $\pi$ drawn from a transition probability kernel $\lambda(\theta,\cdot)$ from $(\rm{\Theta},\mathcal{A})$ into $([0,1],\mathcal{B})$. Then the set $\rmPi = \{(\theta,\pi)\}$ is a Poisson process on $\rm{\Theta} \times [0,1]$ with mean measure $\mu(d\theta)\lambda(\theta,d\pi)$.
\end{lemma}

\begin{lemma}[Superposition property (\cite{Kingman:1993}, Ch. 2, Sec. 2)]\label{lem.superposition}
 Let $\rmPi_1, \rmPi_2,\dots$ be a countable collection of independent Poisson processes on $\rm{\Theta}\times [0,1]$. Let $\rmPi_i$ have mean measure $\nu_i$. Then the superposition $\rmPi = \bigcup_{i=1}^{\infty} \rmPi_i$ is a Poisson process with mean measure $\nu = \sum_{i=1}^{\infty} \nu_i$.
\end{lemma}

In the next section we will present a finite approximation of the beta process and prove that it is asymptotically distributed as a beta process. We will then use this finite approximation in our derivation of the stick-breaking construction presented in Section \ref{sec.stickbreaking}.

\section{A finite approximation of the beta process}\label{sec.finite}

In this section and we consider beta processes with a constant concentration function, $\alpha(\theta) = \alpha$. In this case, the beta process with constant concentration can be approximated using a finite collection of beta random variables and atoms drawn from a base distribution as follows:

\begin{definition}[Beta prior sieves]\label{def.finite_approx}
Let $\alpha(\theta) = \alpha$ and $\mu$ be a diffuse and finite measure on $(\rm{\Theta},\mathcal{A})$. For an integer $K > \mu(\rm{\Theta})$, we define a finite approximation to the beta process as $H_K = \sum_{k=1}^K\pi_k\delta_{\theta_k}$, where $\pi_k \sim Beta(\alpha\mu(\rm{\Theta})/K,\alpha(1-\mu(\rm{\Theta})/K))$ and $\theta_k \sim \mu/\mu(\rm{\Theta})$, with all random variables drawn independently.
\end{definition}

This is similar in spirit to approximations of the Dirichlet process using finite Dirichlet distributions and i.i.d.\ atoms. \cite{Ishwaran:2002a} proved that such an approximation converges in the limit to a Dirichlet process under certain parameterizations. We present a proof of the following corresponding result.

\begin{theorem}[Convergence of finite approximation]\label{thm.finite_convergence}
For the finite approximation of the beta process given in Definition \ref{def.finite_approx}, $\lim_{K\rightarrow\infty} H_K$ converges in distribution to $H \sim \BP(\alpha,\mu)$. 
\end{theorem}

\begin{proof}
 We prove that the Laplace functional of $H_K$ given in Definition \ref{def.finite_approx} converges to the Laplace functional of a beta process given in Equations (\ref{eqn.BP_laplace}) and (\ref{eqn.BP_mean_measure}) with $\alpha$ and $\mu$ satisfying the conditions of Definition \ref{def.finite_approx}.

Let $H:=\lim_{K\rightarrow\infty}H_K$ and let $t < 0$ and $A\in\mathcal{A}$. By the dominated convergence theorem and the independence of all random variables,
\begin{eqnarray}
 \E\,\e^{t H(A)} &=& \lim_{K\rightarrow\infty}\E\,\e^{t H_K(A)}\nn\\
&=&\lim_{K\rightarrow\infty}\left[\E\e^{t\pi\I(\theta\in A)}\right]^K,\label{eqn.limit_proof1}
\end{eqnarray}
where $H_K(A) = \sum_{k=1}^K \pi_k\I(\theta_k\in A)$. Using the tower property of conditional expectation and the law of total probability, the expectation in (\ref{eqn.limit_proof1}) is equal to
\begin{eqnarray}
 \E\,\e^{t\pi\I(\theta\in A)} &=& \E\[\E\[\e^{t\pi\I(\theta\in A)}\I(\theta\in A) + \e^{t\pi\I(\theta\in A)}\I(\theta \not\in A)|\theta\]\]\nn\\
 &=&\mathbb{P}(\theta \in A)\E\,\e^{t\pi} + \mathbb{P}(\theta\not\in A)\nn\\
&=& \frac{\mu(A)}{\mu(\rm{\Theta})}\E\,\e^{t\pi} + 1 - \frac{\mu(A)}{\mu(\rm{\Theta})}.\label{eqn.limit_proof2}
\end{eqnarray}
Again we use the fact that $\theta \sim \mu/\mu(\rm{\Theta})$ and is independent of $\pi$. The Laplace transform of $\pi \sim \rm{Beta}(\alpha\mu(\rm{\Theta})/K,\alpha(1-\mu(\rm{\Theta})/K))$ is
\begin{equation}\label{eqn.limit_proof3}
 \E\,\e^{t\pi} = 1 + \sum_{s=1}^{\infty}\frac{t^s}{s!}\prod_{r=0}^{s-1}\frac{\frac{\alpha\mu(\rm{\Theta})}{K} + r}{\alpha + r}.
\end{equation}
Using this in Equation (\ref{eqn.limit_proof2}) and manipulating the result gives
\begin{eqnarray}
  \E\,\e^{t\pi\I(\theta\in A)} & = & 1 + \frac{\mu(A)}{\mu(\rm{\Theta})}\sum_{s=1}^{\infty}\frac{t^s}{s!}\prod_{r=0}^{s-1}\frac{\frac{\alpha\mu(\rm{\Theta})}{K} + r}{\alpha + r}\\
&=& 1 + \frac{\mu(A)}{K}\sum_{s=1}^{\infty}\frac{t^s}{s!}\prod_{r=1}^{s-1}\frac{r}{\alpha+r} + O(\frac{1}{K^2})\label{eqn.limit_proof4}\\
&=& 1 + \frac{\mu(A)}{K}\sum_{s=1}^{\infty}\frac{t^s}{s!}\frac{\alpha\rmGamma{(\alpha)}\rmGamma{(s)}}{\rmGamma{(\alpha+s)}} + O(\frac{1}{K^2})\\
&=& 1 + \frac{\mu(A)}{K}\sum_{s=1}^{\infty}\frac{t^s}{s!}\int_0^1\alpha\pi^{s-1}(1-\pi)^{\alpha-1}d\pi + O(\frac{1}{K^2})\\
&=& 1 + \frac{\mu(A)}{K}\int_0^1\left(\sum_{s=1}^{\infty}\frac{(t\pi)^s}{s!}\right)\alpha\pi^{-1}(1-\pi)^{\alpha-1}d\pi + O(\frac{1}{K^2})\\
&=& 1 + \frac{\mu(A)}{K}\int_0^1\(\e^{t\pi}-1\)\alpha\pi^{-1}(1-\pi)^{\alpha-1}d\pi + O(\frac{1}{K^2})\label{eqn.limit_proof5}.
\end{eqnarray}
In Equation (\ref{eqn.limit_proof4}) we use the convention that the product equals one when $s=1$. Taking the limit of Equation (\ref{eqn.limit_proof1}) using the value in Equation (\ref{eqn.limit_proof5}), the $O(K^{-2})$ term disappears and we have the familiar exponential limit,
\begin{equation}
 \lim_{K\rightarrow\infty}\left[\E\e^{t\pi\I(\theta\in A)}\right]^K = \exp\left\lbrace \mu(A)\int_0^1\(\e^{t\pi}-1\)\alpha\pi^{-1}(1-\pi)^{\alpha-1}d\pi\right\rbrace ,
\end{equation}
which we recognize as the Laplace transform of a beta process $H(A)$ drawn from a beta process. Since this is true for all $A \in \mathcal{A}$, we get the corresponding Laplace functional of a beta process.
\end{proof}

\section{A stick-breaking construction of the beta process}\label{sec.stickbreaking}
Because the beta process is an infinite jump process, efficient methods are necessary for finding these jump locations. The stick-breaking construction of the beta process is one such method that stands in similar relation to the beta process as the \cite{Sethuraman:1994} construction does to the Dirichlet process. Indeed, because we directly use results from \cite{Sethuraman:1994}, the form of the construction is very similar to Equation (\ref{eqn.DP}).

\begin{theorem}[Stick-breaking construction of the beta process]\label{thm.stickBP}
Let $\mu$ be a diffuse and finite measure on $(\rm{\Theta},\mathcal{A})$ and $\alpha(\theta)$ be a strictly positive and finite function on $\rm{\Theta}$. The following is a constructive definition of the beta process $H\sim\BP(\alpha,\mu)$,
\begin{equation}\label{eqn.Horiginal}
 H = \sum_{i=1}^{\infty} \sum_{j=1}^{C_i} V_{ij}^{(i)}\prod_{l=1}^{i-1}(1-V_{ij}^{(l)})\delta_{\theta_{ij}},
\end{equation}
\begin{equation}\nonumber
C_i \iid \mathrm{Pois}(\mu(\rm{\Theta})),\quad V_{ij}^{(l)}\,|\,\theta_{ij} \ind \mathrm{Beta}(1,\alpha(\theta_{ij})),\quad \theta_{ij} \iid \mu/\mu(\rm{\Theta}).
\end{equation}
\end{theorem}
This construction sequentially incorporates into $H$ a Poisson-distributed number of atoms drawn i.i.d.\ from $\mu/\mu(\rm{\Theta})$, with each group in this sequence indexed by $i$. The atoms receive weights in $(0,1]$ drawn independently as follows: Using an atom-specific stick-breaking construction, an atom in group $i$ throws away the first $i-1$ breaks of its stick and keeps the $i$th break as its weight. We present a straightforward extension of Theorem \ref{thm.stickBP} to $\sigma$-finite $\mu$.
\begin{corollary}[A $\sigma$-finite extension]\label{cor.BPstick}
 Let $(E_k)$ be a partition of $\rm{\Theta}$, where each $E_k$ is compact and therefore $\mu(E_k) < \infty$. The construction of Theorem \ref{thm.stickBP} can be extended to this case by constructing independent beta processes over each $E_k$ and obtaining the full beta process by summing the beta process over each set.
\end{corollary}

For the remainder of the paper we will assume $\mu(\rm{\Theta}) = \gamma < \infty$. Several inference algorithms have been presented for this construction \citep{Paisley:2010,Paisley:2011,Paisley:2012} using both MCMC and variational methods.  Recently, \cite{Broderick:2012b} extended Theorem \ref{thm.stickBP} to beta processes with power-law behavior.
\begin{theorem}[A power-law extension \citep{Broderick:2012b}]\label{prop.powerlaw}
Working within the setup of Theorem \ref{thm.stickBP} with $\alpha(\theta)$ equal to the constant $\alpha$, let $\beta$ be a discount parameter in $(0,1)$. Construct $H$ similar to Theorem \ref{thm.stickBP}, with the exception that $V_{ij}^{(\ell)} \sim Beta(1-\beta,\alpha + i\beta)$. Then $H$ exhibits power law behavior of Type I and II, but not of Type III.
\end{theorem}

\subsection{Proof of Theorem \ref{thm.stickBP} via the finite approximation}\label{sec.stickproof1}
We first prove the construction for constant $\alpha(\theta)$ by constructing finite arrays of random variables and considering their limit. To this end, working with the finite approximation in Definition \ref{def.finite_approx}, we represent each beta-distributed random variable by the stick-breaking construction of \cite{Sethuraman:1994}. That is, we apply the constructive definition of a Dirichlet distribution to the beta distribution, which is the two-dimensional special case. Using this construction, we can draw $\pi \sim \rm{Beta}(a,b)$ as follows.

\begin{lemma}[Constructing a beta random variable \citep{Sethuraman:1994}]\label{lem.construct_beta} 
Draw an infinite sequence of random variables $(V_1,V_2,\dots)$ i.i.d.\ ${Beta}(1,a+b)$ and a second sequence $(Y_1,Y_2,\dots)$ i.i.d.\ ${Bern}(\frac{a}{a+b})$. Construct $\pi = \sum_{i=1}^{\infty} V_i\prod_{j=1}^{i-1}(1-V_j)\mathbb{I}(Y_i = 1)$. Then $\pi$ has a ${Beta}(a,b)$ distribution.
\end{lemma}

Practical applications have led to the study of almost sure truncations of stick-breaking processes \citep{Ishwaran:2001}. By extension, an almost sure truncation of a beta random variable is constructed by truncating the sum in Lemma \ref{lem.construct_beta} at level $R$. As $R\rightarrow\infty$ this truncated random variable converges to a beta-distributed random variable. Using an $R$-truncated beta random variable, a corollary of Theorem \ref{thm.finite_convergence} and Lemma \ref{lem.construct_beta} is,

\begin{corollary}[]\label{cor.trunc_finite_approx}
Under the prior assumptions of Definition \ref{def.finite_approx}, draw two $K\times R$ arrays of independent random variables, $V_{ki} \sim Beta(1,\alpha)$ and $Y_{ki} \sim Bern(\mu(\rm{\Theta})/K)$, and draw $\theta_k$ i.i.d.\ $\mu/\mu(\rm{\Theta})$ for $k=1,\dots,K$. Let
\begin{equation}
H_K^{(R)} = \sum_{i=1}^R \sum_{k=1}^K V_{ki}\prod_{i'=1}^{i-1}(1-V_{ki'})\mathbb{I}(Y_{ki}=1)\delta_{\theta_k}.\nn
\end{equation}
Then $H_K^{(R)}$ converges in distribution to $H\sim\BP(\alpha,\mu)$ by letting $K\rightarrow\infty$ and then $R\rightarrow\infty$.
\end{corollary}

\paragraph{First proof of Theorem \ref{thm.stickBP}.} We show that Theorem \ref{thm.stickBP} with a constant function $\alpha(\theta)$ results from Corollary \ref{cor.trunc_finite_approx} in the limit as $K\rightarrow\infty$ and $R\rightarrow\infty$. We first note that column sums of $Y$ are marginally distributed as $\mbox{Bin}(K,\mu(\rm{\Theta})/K)$, and are independent. This value gives the number of atoms receiving probability mass at step $i$, with $Y_{ki} = 1$ indicating the $k$th indexed atom is one of them. Let the set $\mathcal{I}_i^K = \{k : Y_{ki} = 1\}$ be the index set of these atoms at finite approximation level $K$. This set is constructed by selecting $C^K_i \sim \mbox{Bin}(K,\mu(\rm{\Theta})/K)$ values from $\{1,\dots,K\}$ uniformly without replacement. In the limit $K\rightarrow\infty$, $C_i^K \rightarrow C_i$ with $C_i \sim \mathrm{Pois}(\mu(\rm{\Theta}))$.

Given $k \in \mathcal{I}^K_i$, we know that $\pi_k$ has weight added to it from the $i$th break of its own stick-breaking construction. As a matter of accounting, we are interested other values $i'$ for which $Y_{ki'} = 1$, particularly when $i' < i$. We next show that in the limit $K\rightarrow\infty$, the index values in the set $\mathcal{I}_i := \mathcal{I}^{\infty}_i$ are always unique from those in previous sets ($i' < i$), meaning for a given column $L\leq R$, we see new index values with probability equal to one. We are therefore always adding probability mass to new atoms. Let $E$ be the event that there exists a number $k \in \mathcal{I}_i\cap\mathcal{I}_{i'}$ for $i\neq i'$ and $i,i' \leq L\leq R$. We can bound the probability of this event as follows:
\begin{eqnarray}
\mathrm{P}_K(\textstyle\bigcup_{i'<i\leq L}\mathcal{I}_i^K\cap\mathcal{I}_{i'}^K \neq \emptyset\, |\,\mu) &\leq & \sum\limits_{i'<i\leq L} \mathrm{P}_K(\mathcal{I}_i^K\cap\mathcal{I}_{i'}^K \neq \emptyset\,|\,\mu)\nn\\
&\leq & \sum_{i'<i\leq L}\sum_{k=1}^K\mathrm{P}_K(Y_{ki}Y_{ki'} = 1|\mu)\nn\\
&\leq & \frac{L(L-1)}{2}\frac{\mu(\rm{\Theta})^2}{K}.
\end{eqnarray}
Therefore, for any finite integer $L\leq R$, in the limit $K\rightarrow\infty$ the atoms $\theta_k$, $k\in \mathcal{I}_L$, are different from all previously observed atoms with probability one since $\mu$ is a diffuse measure. Since this doesn't depend on $R$, we can let $R\rightarrow\infty$. The proof concludes by recognizing that the resulting process is equivalent to (\ref{eqn.Horiginal}) $\hfill\square$

\subsection{Proof of Theorem \ref{thm.stickBP} via Poisson processes}\label{sec.BPproof}

In this section, we give a second proof based on the the Poisson process that is extended to a non-constant $\alpha(\theta)$. Specifically, we show that the construction of Theorem \ref{thm.stickBP} has the distribution of a beta process by showing that its Laplace functional has the form given in Equations (\ref{eqn.BP_laplace}) and (\ref{eqn.BP_mean_measure}). To this end, we use the following lemma to obtain an equivalent representation of Theorem \ref{thm.stickBP}.
\begin{lemma}\label{lem.stick_exponential}
Let $(V_1,\dots,V_r)$ be i.i.d.\ $\mathrm{Beta}(1,\alpha)$. If $T \sim \mathrm{Gamma}(r,\alpha)$, then the random variables $\prod_{j=1}^r(1-V_j)$ and $\exp\{-T\}$ are equal in distribution.
\end{lemma}
\begin{proof}
 Define $\xi_k = -\ln(1-V_k)$. By a change of variables, $\xi_k \sim \mathrm{Exp}(\alpha)$. The function $-\ln \prod_{j=1}^r(1-V_j) = \sum_{k=1}^r \xi_k$, where the $\xi_k$ are i.i.d.\ because $V_k$ are i.i.d. Therefore $T := \sum_{k=1}^r\xi_k$ has a $\mathrm{Gam}(r,\alpha)$ distribution and the result follows.
\end{proof}
Using Lemma \ref{lem.stick_exponential}, we have a construction of $H$ equivalent to that given in Theorem \ref{thm.stickBP},
\begin{equation}\label{eqn.H}
 H = \sum_{j=1}^{C_1}V_{1j}\delta_{\theta_{1j}} + \sum_{i=2}^{\infty} \sum_{j=1}^{C_i} V_{ij}\mathrm{e}^{-T_{ij}}\delta_{\theta_{ij}},\quad C_i \iid \mathrm{Pois}(\mu(\rm{\Theta})),
\end{equation}
\begin{equation}\nonumber
V_{ij}\,|\,\theta_{ij} \iid \mathrm{Beta}(1,\alpha(\theta_{ij})),\quad  T_{ij}\,|\,\theta_{ij} \stackrel{ind}{\sim} \mathrm{Gam}(i-1,\alpha(\theta_{ij})),\quad \theta_{ij} \iid \mu/\mu(\rm{\Theta}).
\end{equation}
\paragraph{Second proof of Theorem \ref{thm.stickBP}.} By applying Lemmas \ref{lem.marked} and \ref{lem.superposition}, we observe that the construction in (\ref{eqn.H}) has an underlying Poisson process as follows: Let $H_i := \sum_{j=1}^{C_i} \pi_{ij}\delta_{\theta_{ij}}$ and $H = \sum_{i=1}^{\infty} H_i$, where $\pi_{1j} = V_{1j}$ and $\pi_{ij} := V_{ij}\e^{-T_{ij}}, i > 1$. The set of atoms in each $H_i$ forms an independent Poisson process $\rmPi_i^*$ on $\rm{\Theta}$ with mean measure $\mu$. The atoms in Poisson process $\rmPi_i^*$ are marked with weights $\pi \in [0,1]$ conditioned on $\theta$ that are independent $\lambda_i$, where $\lambda_i$ is the probability measure on $\pi_{ij}$, which we derive later.  

It follows from Lemma \ref{lem.marked} that each $H_i$ is a function of an Poisson process $\rmPi_i = \{(\theta_{i},\pi_{i})\}$ on $\rm{\Theta} \times [0,1]$ with mean measure $\mu\times \lambda_i$. Therefore, by Lemma \ref{lem.superposition} $H$ is a function of a Poisson process $\rmPi = \bigcup_{i=1}^{\infty} \rmPi_i$ with mean measure $\nu = \sum_{i=1}^{\infty}\nu_i = \mu\times\sum_{i=1}^{\infty} \lambda_i$. We see that calculating $\nu$ for (\ref{eqn.H}) amounts to summing the L\'{e}vy measures $\lambda_i$. These measures fall into two cases, which we give below.

\emph{Case} $i=1:$~ Since $V_{1j}\,|\,\theta_{1j} \sim \mathrm{Beta}(1,\alpha(\theta_{1j}))$, the Poisson process $\rmPi_1$ underlying $H_1$ has mean measure $\mu\times\lambda_1$, with L\'{e}vy measure
\begin{equation}\label{eqn.f_1}
 \lambda_1(d\pi) = \alpha(\theta) (1-\pi)^{\alpha(\theta)-1}d\pi.
\end{equation}
We write $\lambda_1(d\pi) = f_1(\pi|\alpha,\theta)d\pi$ where $f_1(\pi|\alpha,\theta)$ is the density of $\nu_1$ with respect to Lebesgue measure $d\pi$.

\emph{Case} $i > 1:$~ The L\'{e}vy measure of the Poisson process $\rmPi_i$ underlying $H_i$ for $i > 1$ requires more work to derive its associated density $f_i$. Recall that $\pi_{ij} := V_{ij}\exp\{-T_{ij}\}$, where $V_{ij}\,|\,\theta_{ij} \sim \mathrm{Beta}(1,\alpha(\theta_{ij}))$ and $T_{ij}\,|\,\theta_{ij} \sim \mathrm{Gamma}(i-1,\alpha(\theta_{ij}))$. First, let $W_{ij} := \exp\{-T_{ij}\}$. Then by a change of variables,
$p_W(w|i,\alpha,\theta) = \frac{\alpha(\theta)^{i-1}}{(i-2)!} w^{\alpha(\theta)-1}(-\ln w)^{i-2}.$
Using the product distribution formula for two random variables \citep{Rohatgi:1976}, the density of $\pi_{ij} = V_{ij}W_{ij}$ is
\begin{eqnarray}\label{eqn.f_i}
f_i(\pi|\alpha,\theta) \hspace{-2mm}&=&\hspace{-2mm} \int_{\pi}^1 w^{-1}p_V(\pi/w|\alpha,\theta)p_W(w|i,\alpha,\theta)dw\\
\hspace{-2mm}&=&\hspace{-2mm} \frac{\alpha(\theta)^{i}}{(i-2)!}\int_{\pi}^1 w^{-1}(\ln \frac{1}{w})^{i-2}(w-\pi)^{\alpha(\theta)-1} dw\nonumber.
\end{eqnarray}
This integral does not have a closed-form solution. 

\emph{Calculating $\lambda:$}~ We have decomposed the $\rm{\Sigma}$-finite measure $\lambda$ into a sequence of finite measures that can be added to calculate the mean measure of the Poisson process underlying (\ref{eqn.Horiginal}). Since
$$\nu(d\theta,d\pi) = \sum_{i=1}^{\infty} (\mu\times \lambda_i)(d\theta, d\pi) = \mu(d\theta)d\pi\sum_{i=1}^{\infty} f_i(\pi|\alpha,\theta),$$
by showing that $\sum_{i=1}^{\infty} f_i(\pi|\alpha,\theta) = \alpha(\theta)\pi^{-1}(1-\pi)^{\alpha(\theta)-1}$, we complete the proof. From Equations (\ref{eqn.f_1}) and (\ref{eqn.f_i}) we have that that $\lambda(d\pi) = \alpha(\theta)(1-\pi)^{\alpha(\theta)-1}d\pi + \sum_{i=2}^{\infty}f_i(\pi|\alpha,\theta)d\pi$, where
\begin{eqnarray}\label{eqn.f_seq}
 \sum_{i=2}^{\infty} f_i(\pi|\alpha,\theta) &=& \sum_{i=2}^{\infty} \frac{\alpha(\theta)^{i}}{(i-2)!}\int_{\pi}^1 w^{\alpha(\theta)-2}(\ln \frac{1}{w})^{i-2}(1-\frac{\pi}{w})^{\alpha(\theta)-1} dw\nonumber\\
&=&\hspace{-2mm}\alpha(\theta)^2\int_{\pi}^1  w^{\alpha(\theta)-2}(1-\frac{\pi}{w})^{\alpha(\theta)-1}dw \sum_{i=2}^{\infty}\frac{\alpha(\theta)^{i-2}}{(i-2)!}(\ln \frac{1}{w})^{i-2}\nonumber\\
&=&\hspace{-2mm}\alpha(\theta)^2\int_{\pi}^1  w^{-2}(1-\pi/w)^{\alpha(\theta)-1}dw\,.
\end{eqnarray}
The second equality is by monotone convergence and Fubini's theorem and leads to an exponential power series. The last integral is equal to $\frac{\alpha(\theta)(1-\pi)^{\alpha(\theta)}}{\pi}$. Adding the two terms shows that the mean measure equals that of a beta process given in Equation (\ref{eqn.BP_mean_measure}). \hfill $\square$

\section{Almost sure truncations of the beta process}\label{sec.truncations}
Truncated beta processes can be used in MCMC sampling schemes, and also arise in the variational inference setting 
\citep{Teh:2009,Paisley:2011,Wainwright:2008}. Poisson process 
representations are useful for characterizing the part of the beta process 
that is being thrown away in the truncation. Consider a stick-breaking construction of the beta process 
truncated after group $R$, defined as $H^{(R)} = \sum_{i=1}^R H_i$. 
The part being discarded, $H - H^{(R)}$, has an underlying Poisson process 
with mean measure 
\begin{equation}
\nu_R^+(d\theta,d\pi)  := \sum_{i=R+1}^{\infty} \nu_i(d\theta,d\pi) =  \mu(d\theta)\times\sum_{i=R+1}^{\infty} \lambda_i(d\pi),
\end{equation}
and a corresponding counting measure $N_R^+(d\theta,d\pi)$. This measure 
contains information about the missing atoms; for example, the number of 
missing atoms having weight $\pi \geq \epsilon$ is Poisson distributed with 
parameter $\nu_R^+(\rm{\Theta},[\epsilon,1])$.

For truncated beta processes, a measure of closeness to the true beta process is helpful when selecting truncation levels. We derive approximation error bounds in the context of the beta-Bernoulli process. The definition of the Bernoulli process is
\begin{definition}[Beta-Bernoulli process]\label{def.BeP}
Draw a beta process $H \sim \BP(\alpha,\mu)$ on $(\rm{\Theta},\mathcal{A})$ with $\mu$ finite. Define a process $X$ on the atoms of $H$ such that $X(\{\theta\})|H \sim \mbox{Bern}(H(\{\theta\}))$ independently for all $\theta \in H$. Then $X$ is a Bernoulli process, denoted $X\,|\,H \sim \mbox{BeP}(H)$. 
\end{definition}

Returning to the bound, let data $Y_n \sim f(X_n,\phi_n)$, where $X_n$ is a Bernoulli process taking either $H$ or $H^{(R)}$ as parameters, and $\phi_n$ is a set of additional parameters (which could be globally shared). Let $\Y = (Y_1,\dots,Y_M)$. One measure of closeness is the total variation distance between the marginal density of $\Y$ under the beta process, denoted $\m_{\infty}(\Y)$, and the process truncated at group $R$, denoted $\m_R(\Y)$. This measure originated with work on truncated Dirichlet processes in \cite{Ishwaran:2001} and was extended to the beta process in \cite{Teh:2009}. 

After slight modification to account for truncating groups rather than atoms, we have
\begin{equation}\label{eqn.bound}
\frac{1}{2}\int |\m_R(\Y) - \m_{\infty}(\Y)|d\Y  \leq  \mathbb{P}\left\lbrace\exists (i,j), i > R, 1 \leq n \leq M : X_n(\theta_{ij}) \neq 0\right\rbrace.
\end{equation}
We derive this bound in the appendix. In words, this says that one half the 
total variation between $\m_R$ and $\m_{\infty}$ is less than one minus the 
probability that, in $M$ Bernoulli processes with parameter 
$H\sim\mathrm{BP}(\alpha,\mu)$, $X_n(\theta) = 0$ for all $\theta \in H_i$ when $i> R$. 
In \cite{Teh:2009} and \cite{Paisley:2011}, a looser version of this bound was obtained. Using the Poisson process 
representation of $H$, we can give an exact form of this bound. 

\begin{theorem}[Truncated stick-breaking constructions]\label{thm.truncations}
Let $X_{1:M} \iid \mathrm{BeP}(H)$ with $H \sim \mathrm{BP}(\alpha,\mu)$ constructed as in (\ref{eqn.Horiginal}). For a truncation level $R$, let $E$ be the event that there exists an index $(i,j)$ with $i>R$ such that $X_n(\theta_{ij}) = 1$. Then the bound in (\ref{eqn.bound}) equals
$$\mathbb{P}(E) = 1 - \exp\left\lbrace - \int_{(0,1]} \nu_R^+(\rm{\Theta},d\pi)\left(1 - (1-\pi)^M\right)\right\rbrace.$$
\end{theorem}

\paragraph{\emph{Proof}}(Simple functions)
Let the set $B_{nk} = \left[\frac{k-1}{n},\frac{k}{n}\right)$ and $b_{nk} = \frac{k-1}{n}$, where $n$ and $k \leq n$ are positive integers.  Approximate the variable $\pi\in [0,1]$ with the simple function $g_n(\pi) = \sum_{k=1}^n b_{nk}\boldsymbol{1}_{B_{nk}}(\pi)$. We calculate the truncation error term, $\mathbb{P}(E^c)=\E[\prod_{i>R,j}(1-\pi_{ij})^M]$, by approximating with $g_n$, re-framing the problem as a Poisson process with mean and counting measures $\nu_R^+$ and $N_R^+(\rm{\Theta},B)$, and then taking a limit:
\begin{eqnarray}
\E\left[\prod\nolimits_{i>R,j}(1-\pi_{ij})^M\right]& = & \hspace{-1mm}\lim_{n\rightarrow\infty} \prod_{k=2}^n\E\left[(1-b_{nk})^{M\cdot N_R^+(\rm{\Theta},B_{nk})}\right]\\
& = & \hspace{-1mm}\exp\left\lbrace \lim_{n\rightarrow\infty} - \sum_{k=2}^n \nu_R^+(\rm{\Theta},B_{nk})\left(1-(1-b_{nk})^M\right)\right\rbrace .\nonumber 
\end{eqnarray}
For a fixed $n$, this approach divides the interval $[0,1]$ into disjoint regions that can be analyzed separately as independent Poisson processes. Each region uses the approximation $\pi \approx g_n(\pi)$, with  $\lim_{n\rightarrow\infty}g_n(\pi) = \pi$, and $N_R^+(\rm{\Theta},B)$ counts the number of atoms with weights that fall in the interval $B$. Since $N_R^+$ is Poisson distributed with mean $\nu_R^+$, the expectation follows.$\hfill\square$\newline

One can use approximating simple functions to give an arbitrarily close approximation of Theorem 3; since $\nu_R^+ = \nu_{R-1}^+ - \nu_{R}$ and $\nu_0^+ = \nu$, performing a sweep of truncation levels requires approximating only one additional integral for each increment of $R$. From the Poisson process, we also have the following analytical bound, which we present for a constant $\alpha(\theta) = \alpha$.
\begin{corollary}[A bound on Theorem \ref{thm.truncations}]\label{cor.truncations}
Given the setup in Theorem \ref{thm.truncations} with $\alpha(\theta) = \alpha$, an upper bound on $\mathbb{P}(E)$ is $$\mathbb{P}(E) \leq 1 - \exp\left\lbrace-\mu(\mathrm{\Theta}) M \left(\frac{\alpha}{1+\alpha}\right)^R\right\rbrace.$$
\end{corollary}
\begin{proof}
 From the proof of Theorem \ref{thm.truncations}, we have $$\textstyle \mathbb{P}(E) = 1 - \E[\prod_{i>R,j}(1-\pi_{ij})^M] \leq 1 - \E[\prod_{i>R,j}(1-\pi_{ij})]^M.$$ This second expectation can be calculated as in Theorem \ref{thm.truncations} with $M$ replaced by a one. We therefore wish to calculate the negative of $\textstyle \int_0^1 \pi\nu_R^+(\rm{\Theta},d\pi)$. Let $q_r$ be the distribution of the $r$th break from a $\mathrm{Beta}(1,\alpha)$ stick-breaking process. Then this integral is equal to $\mu(\mathrm{\Theta}) \sum_{r=R+1}^{\infty}\E_{q_r}[\pi].$ Since $\E_{q_r}[\pi]=\alpha^{-1}(\frac{\alpha}{1+\alpha})^r$, this solves to $\mu(\mathrm{\Theta}) (\frac{\alpha}{1+\alpha})^R$.
\end{proof}
We observe that the term in the exponential equals the negative of $M\int_0^1 \pi \nu_R^+(\rm{\Theta},d\pi)$, which is the expected number of missing ones in $M$ observations from the truncated Bernoulli process. Other stochastic processes can also take $H$ as parameter. The negative binomial process is one such process \citep{Broderick:2012a,Zhou:2012c,Heaukulani:2013} and has the following definition.

\begin{definition}[Beta-Negative binomial process]\label{def.BeP}
Draw a beta process $H \sim \BP(\alpha,\mu)$ on $(\rm{\Theta},\mathcal{A})$ with $\mu$ finite. Define a process $X$ on the atoms of $H$ such that $X(\{\theta\})|H \sim \mbox{NegBin}(r,H(\{\theta\}))$ independently for all $\theta \in H$, where $r > 0$. Then $X$ is a negative binomial process, denoted $X\,|\,H \sim \mbox{NBP}(H)$. 
\end{definition}

In the derivation in the appendix, we show how the inequality in Equation (\ref{eqn.bound}) applies to the beta-negative binomial process as well. The only difference is in calculating the probability of the event we call $E$ above. This results in the following.

\begin{corollary}[A negative binomial extension]\label{cor.truncations2}
Given the setup in Theorem \ref{thm.truncations}, but with $X_{1:M} \iid \mathrm{NBP}(r,H)$ we have
\begin{eqnarray}
\mathbb{P}(E) &=& 1 - \exp\left\lbrace - \int_{(0,1]} \nu_R^+(\rm{\Theta},d\pi)\left(1 - (1-\pi)^{Mr}\right)\right\rbrace,\nn\\
 \mathbb{P}(E) &\leq & 1 - \exp\left\lbrace-\mu(\mathrm{\Theta}) Mr \left(\frac{\alpha}{1+\alpha}\right)^{R}\right\rbrace.\nn
\end{eqnarray}
\end{corollary}
\begin{proof}
After recognizing that $\textstyle \mathbb{P}(E^c) = \E[\prod_{i>R,j}(1-\pi_{ij})^{Mr}] \geq \E[\prod_{i>R,j}(1-\pi_{ij})]^{Mr}$ for this problem, proof of the first line follows from the proof of Theorem \ref{thm.truncations} and proof of the second line follows from the proof of Corollary \ref{cor.truncations}.
\end{proof}

\section{Posterior inference for the beta process}\label{sec.MCMC}
We present a simple method for sampling from the posterior of the beta-Binomial and beta-negative binomial processes. We show that posterior sampling can be separated into two parts: Sampling the weights of the almost surely finite number of observed atoms and sampling the weights of the infinitely remaining unobserved atoms. To this end, we will need the following lemma for sampling from a beta distribution.
\begin{lemma}[Product representation of a beta random variable]\label{lem.prod_beta}
Let $\eta_1\sim Beta(a_1,a_2)$, $\eta_2\sim Beta(b_1,b_2)$ and $\eta_3 \sim Beta(a_1+a_2,b_1+b_2)$, all independently. If we define the weighted average $\pi = \eta_3 \eta_1 + (1-\eta_3)\eta_2$, then $\pi \sim Beta(a_1+b_1,a_2+b_2)$. 
\end{lemma}
\begin{proof}
 This is a special case of Lemma 3.1 in \cite{Sethuraman:1994}.
\end{proof}

\subsection{Sampling from the posterior of the beta-Bernoulli process} 
The beta process is conjugate to the Bernoulli process \citep{Kim:1999}, which is evident using the following hierarchical representation for the beta-Bernoulli process (see the appendix for details)
\begin{equation}
X_i(d\theta)\, |\, H \stackrel{iid}{\sim} \rm{Bern}(H(d\theta)),\quad H(d\theta) \sim \rm{Beta}\{\alpha(\theta)\mu(d\theta),\alpha(\theta)(1-\mu(d\theta))\}. 
\end{equation}
Let the count statistics from $n$ independent Bernoulli processes be
$$M_1(d\theta) := \sum_{i=1}^n X_i(d\theta),\quad M_0(d\theta) := \sum_{i=1}^n \(1 - X_i(d\theta)\).$$
Using the count functions $M_1$ and $M_0$, the posterior distribution of $H$ is
\begin{equation}\label{eqn.Hpost}
H(d\theta)\,|\,X_1,\dots,X_n \sim \rm{Beta}\{\alpha(\theta)\mu(d\theta) + M_1(d\theta),\alpha(\theta)(1-\mu(d\theta)) + M_0(d\theta)\}.
\end{equation}
In light of Lemma \ref{lem.prod_beta}, we can sample from the distribution in Equation (\ref{eqn.Hpost}) as follows,
\begin{equation}
H(d\theta)\,|\,X_1,\dots,X_n = \eta(d\theta)P(d\theta) + (1-\eta(d\theta))H'(d\theta),\nn
\end{equation}
\begin{equation}
\eta(d\theta) \ind \rm{Beta}(n,\alpha(\theta)),\quad P(d\theta) \ind \rm{Beta}\(M_1(d\theta),M_0(d\theta)\),\quad H' \ind \rm{BP}(\alpha,\mu).
\end{equation}

Therefore, drawing from the posterior involves sampling from an uncountably infinite set of beta distributions. Fortunately we can separate this into an observed countable set of atoms, and the remaining unobserved and uncountable locations \citep{Kingman:1967}. Let $\rm{\Omega} = \{\theta : M_1(d\theta) > 0\}$. Then $\rm{\Omega}$ is almost surely finite \citep{Thibaux:2007} and can be referred to as the set of ``observed'' atoms. We note that for all $\theta \not\in \rm{\Omega}$, $P(d\theta) = 0$ a.s., whereas $P(d\theta) > 0$ a.s.\ if $\theta \in \rm{\Omega}$. $H'(d\theta) = 0$ a.s. for all $\theta \in \rm{\Omega}$ since $\mu$ is diffuse. To sample from this posterior, the weighted probabilities $\eta(d\theta)P(d\theta)$ can be drawn first for all $\theta \in \rm{\Omega}$. Then $H'$ can be drawn using a truncated stick-breaking construction of the beta process following which an atom $\theta \in H'$ is weighted by $1-\eta(\theta) \sim \rm{Beta}(\alpha(\theta),n)$ This can be drawn separately since an atom 
in $H'$ is a.s.\ not in $\rm{\Omega}$.\footnote{A similar posterior sampling method can be used for Dirichlet processes as well.}

\subsection{Sampling from the posterior of the beta-NB process} 
Sampling from the posterior of the beta-negative Binomial process follows a procedure almost identical to the beta-Bernoulli process. In this case the hierarchical process is
\begin{equation}
X_i(d\theta) | H \stackrel{iid}{\sim} \rm{NegBin}(r,H(d\theta)),\quad H(d\theta) \sim \rm{Beta}\{\alpha(\theta)\mu(d\theta),\alpha(\theta)(1-\mu(d\theta))\}. 
\end{equation}
Therefore, the posterior distribution is
\begin{equation}\label{eqn.Hpost2}
H(d\theta)\,|\,X_1,\dots,X_n \sim \rm{Beta}\{\alpha(\theta)\mu(d\theta) + \textstyle\sum_i X_i(d\theta),\alpha(\theta)(1-\mu(d\theta)) + nr\}.
\end{equation}
Again using Lemma \ref{lem.prod_beta}, we can sample from the distribution in Equation (\ref{eqn.Hpost2}) as follows,
\begin{equation}
H(d\theta)\,|\,X_1,\dots,X_n = \eta(d\theta)P(d\theta) + (1-\eta(d\theta))H'(d\theta),
\end{equation}
\begin{equation}
\eta(d\theta) \ind \rm{Beta}(\textstyle\sum_i X_i(d\theta)+nr,\alpha(\theta)),~ P(d\theta) \ind \rm{Beta}\(\textstyle\sum_i X_i(d\theta),nr\),~ H' \ind \rm{BP}(\alpha,\mu).\nn
\end{equation}

As with the beta-Bernoulli process, we can separate this into sampling the observed and unobserved atoms; an atom $\theta$ is again considered to be observed and contained in the set $\rm{\Omega}$ if $\sum_i X_i(d\theta) > 0$ and $\rm{\Omega}$ is again a.s.\ finite \citep{Broderick:2012a}. A sample from the posterior of $H$ can be constructed by adding delta measures $\eta(d\theta_k)P(d\theta_k)\delta_{\theta_k}$ at each $\theta_k \in \rm{\Omega}$. We construct the remaining atoms by sampling $H'$ from a truncated stick-breaking construction of the beta process and down-weighting the measure on an atom $\theta \in H'$ by multiplying $H'(d\theta)$ with a $\rm{Beta}(\alpha(\theta),nr)$ random variable.

\section{Conclusion}\label{sec.conclusion}
We have presented a constructive definition of the beta process.  We showed how this construction follows naturally by using the special case of the \cite{Sethuraman:1994} construction applied to the beta distribution. To this end, we presented a finite approximation of the beta process and proved that it has the correct limiting distribution. We also proved and gave further truncation analysis of the construction using the Poisson process theory. In the final section, we showed how posterior sampling of the beta process can be done easily by separating the atomic from the non-atomic parts of the posterior distribution. This produces a set of atoms whose weights can be sampled from an atom-specific beta posterior distribution, and a sample from the beta process prior using the stick-breaking construction that is subsequently down-weighted.

\bibliographystyle{sjs}
\bibliography{BP_stick}

\section*{Appendix}

\subsection*{A1. Derivation of the almost sure truncation bound}

We derive the inequality of Equation (\ref{eqn.bound}) by calculating the total variation distance between the marginal distributions of data $\Y := Y_1,\dots,Y_n$ under a beta-Bernoulli and beta-negative binomial process, and those same processes with the stick-breaking construction for the underlying beta process truncated after group $R$, denoted $H^{(R)}$. We write these marginals as $\m_{\infty}(\Y)$ and $\m_{R}(\Y)$ respectively. 

For the sequence of Bernoulli or negative binomial processes $\X := X_1,\dots,X_n$, we let $\boldsymbol{\pi}_{\infty}(\X)$ be the marginal distribution of the selected process under a beta process prior, and $\boldsymbol{\pi}_{R}(\X)$ the corresponding marginal under an $R$-truncated beta process. The domain of integration for $\X$ is $\{0,1\}^{\infty\times n}$ for the Bernoulli process and $\{0,\infty\}^{\infty\times n}$ for the negative binomial process. We write this generically as $\{0,y\}^{\infty\times n}$ below. The total variation of these two marginal distributions can be bounded as follows,
\begin{eqnarray}\label{eqn.TV1}
 \int_{\rm{\Omega}_{\Y}}|\m_{\infty}(\Y)-\m_{R}(\Y)|d\Y &=& \int_{\rm{\Omega}_{\Y}} \left|\sum_{\X\in\{0,y\}^{\infty \times n}}\prod_{i=1}^n f(Y_i|X_i)\{\boldsymbol{\pi}_{R}(\X)-\boldsymbol{\pi}_{\infty}(\X)\}\right|d\Y\nn\\
 &\leq& \sum_{\X\in\{0,y\}^{\infty \times n}} \int_{\rm{\Omega}_{\Y}}\prod_{i=1}^n f(Y_i|X_i)d\Y \left|\boldsymbol{\pi}_{R}(\X)-\boldsymbol{\pi}_{\infty}(\X)\right|\nn\\
 &=& \sum_{\X\in\{0,y\}^{\infty \times n}}\left|\boldsymbol{\pi}_{R}(\X)-\boldsymbol{\pi}_{\infty}(\X)\right|.
\end{eqnarray}
We observe that the rows of $\X$ are independent under both priors. Let $\X_R$ be the first $\tilde{R} := \sum_{i=1}^R C_i$ rows of $\X$, which is random and a.s.\ finite, and let $\X_{R+}$ be the remaining rows. Under the truncated prior, $\boldsymbol{\pi}_R(\X_{R+}=\boldsymbol{0}) = 1$, while the two processes share the same measure for $\X_R$, $\boldsymbol{\pi}_{R}(\X_R)=\boldsymbol{\pi}_{\infty}(\X_R)$. The sequence in (\ref{eqn.TV1}) continues as follows,
\begin{eqnarray}
 \sum_{\X}\left|\boldsymbol{\pi}_{R}(\X)-\boldsymbol{\pi}_{\infty}(\X)\right| &=& \sum_{\substack{\X_{R}\in\{0,y\}^{\tilde{R} \times n} \\ \X_{R+}\in\{0,y\}^{\infty \times n}}}\left|\boldsymbol{\pi}_{R}(\X_R)\boldsymbol{\pi}_{R}(\X_{R+})-\boldsymbol{\pi}_{\infty}(\X_R)\boldsymbol{\pi}_{\infty}(\X_{R+})\right|\nn\\
 &=&\sum_{\substack{\X_{R}\in\{0,y\}^{\tilde{R} \times n} \\ \X_{R+} = \boldsymbol{0}}}\left|\boldsymbol{\pi}_{R}(\X_R)-\boldsymbol{\pi}_{\infty}(\X_R)\boldsymbol{\pi}_{\infty}(\X_{R+})\right|\nn\\
 &&+ \sum_{\substack{\X_{R}\in\{0,y\}^{\tilde{R} \times n} \\ \X_{R+} \neq \boldsymbol{0}}} \boldsymbol{\pi}_{\infty}(\X_{R})\boldsymbol{\pi}_{\infty}(\X_{R+})\nn\\
 &=& 1 - \boldsymbol{\pi}_{\infty}(\X_{R+}=\boldsymbol{0}) + \boldsymbol{\pi}_{\infty}(\X_{R+}\neq\boldsymbol{0})\nn\\
 &=& 2(1-\boldsymbol{\pi}_{\infty}(\X_{R+}=\boldsymbol{0})).
\end{eqnarray}
The result follows by observing that this is two times the probability of the event considered in Equation (\ref{eqn.bound}).

\subsection*{A2. An alternate definition of the beta process}

Though the form of the L\'{e}vy measure $\lambda$ in Equation (\ref{eqn.BP_mean_measure}) is suggestive of a beta distribution, the following second definition of the beta process makes the relationship of the beta process to the beta distribution more explicit in the infinitesimal case.

\begin{definition}[The beta process II]\label{def.bp2}
Let $\mu$ be a diffuse $\sigma$-finite measure on $(\rm{\Theta},\mathcal{A})$ and let $\alpha(\theta)$ be a finite strictly positive function on $\rm{\Theta}$. For all infinitesimal sets $d\theta \in \mathcal{A}$, let
$$H(d\theta) \ind \rm{Beta}\{\alpha(\theta)\mu(d\theta),\alpha(\theta)(1-\mu(d\theta))\}.$$
Then $H$ is a beta process, denoted $H \sim \BP(\alpha,\mu)$.
\end{definition}

It isn't immediately obvious that Definitions \ref{def.bp1} and \ref{def.bp2} are of the same stochastic process. This was proved by \cite{Hjort:1990} in the context of survival analysis, where $\mu$ was a measure on $\R_+$, and the proof there heavily relied on this one-dimensional structure. We give a proof of this equivalence, stated in the following theorem, using more general spaces.

\begin{theorem}
 Definitions \ref{def.bp1} and \ref{def.bp2} are of the same stochastic process.
\end{theorem}

\begin{proof}
We prove the equivalence of Definitions \ref{def.bp1} and \ref{def.bp2} by showing that the the Laplace transform of $H(A)$ from Definition \ref{def.bp2} has the form of Equation (\ref{eqn.BP_laplace}) with the mean measure given in Equation (\ref{eqn.BP_mean_measure}). Since their Laplace transforms are equal, the equivalence follows. The proof essentially follows the same pattern as the proof of Theorem \ref{thm.finite_convergence} with a change in notation to account for working with measures on infinitesimal sets as opposed to an asymptotic analysis of a finite approximation. 

Let $H$ be as in Definition \ref{def.bp2}. We calculate the Laplace transform of $H(A) = \int_{d\theta\in A} H(d\theta)$, where by definition $H(d\theta)$ is beta-distributed with parameters $\alpha(\theta)\mu(d\theta)$ and $\alpha(\theta)(1-\mu(d\theta))$. For $t < 0$, the Laplace transform of $H(A)$ can be written as
\begin{eqnarray}
 \E\,\e^{\int_{d\theta\in A} tH(d\theta)} &=& \prod_{d\theta\in A} \E\,\e^{tH(d\theta)}\label{eqn.bp_proof0}\\
&=& \prod_{d\theta\in A}\(1 + \sum_{k=1}^{\infty}\frac{t^k}{k!}\prod_{r=0}^{k-1}\frac{\alpha(\theta) \mu(d\theta) + r}{\alpha(\theta) + r}\).\label{eqn.bp_proof1}
\end{eqnarray}
The first equality uses the property that an exponential of a sum factorizes into a product of exponentials, which has the given product integral extension. The independence in Definition \ref{def.bp2} allows for the expectation to be brought inside the product. The second equality is the Laplace transform of a beta random variable with the given parameterization.

The remainder is simply a manipulation of Equation (\ref{eqn.bp_proof1}) until it reaches the desired form. We first present the sequence of equalities, followed by line-by-line explanations. Continuing from Equation (\ref{eqn.bp_proof1}),
\begin{eqnarray}
 \E\,\e^{H(A)} &=& \prod_{d\theta\in A}\(1 + \mu(d\theta)\sum_{k=1}^{\infty}\frac{t^k}{k!}\prod_{r=1}^{k-1}\frac{r}{\alpha(\theta) + r}\)\label{eqn.bp_proof2}\\
&=& \prod_{d\theta\in A} \(1 + \alpha(\theta)\mu(d\theta)\sum_{k=1}^{\infty}\frac{t^k}{k!} \frac{\rmGamma(k)\rmGamma(\alpha(\theta))}{\rmGamma(\alpha(\theta)+k)}\)\label{eqn.bp_proof3}\\
&=& \prod_{d\theta\in A} \(1 + \alpha(\theta)\mu(d\theta)\sum_{k=1}^{\infty}\frac{t^k}{k!}\int_0^1 \pi^{k-1}(1-\pi)^{\alpha(\theta)-1}d\pi\)\label{eqn.bp_proof4}\\
&=& \prod_{d\theta\in A} \(1 + \alpha(\theta)\mu(d\theta)\int_0^1\(\textstyle\sum_{k=1}^{\infty}\frac{(t\pi)^k}{k!}\) \pi^{-1}(1-\pi)^{\alpha(\theta)-1}d\pi\)\label{eqn.bp_proof5}\\
&=& \prod_{d\theta\in A} \(1 + \alpha(\theta)\mu(d\theta)\int_0^1(\e^{t\pi}-1) \pi^{-1}(1-\pi)^{\alpha(\theta)-1}d\pi\).\label{eqn.bp_proof6}
\end{eqnarray}
We derive this sequence as follows: Equation (\ref{eqn.bp_proof2}) uses the fact that $\mu$ is a diffuse measure---we can pull out the $r=0$ term and disregard any terms with $\mu(d\theta)^k$ for $k > 1$ since this integrates to zero; Equation (\ref{eqn.bp_proof3}) uses the equality $\rmGamma(c+1) = c\rmGamma(c)$; Equation (\ref{eqn.bp_proof4}) recognizes the fraction of gamma functions as the normalizing constant of a beta distribution with parameters $k$ and $\alpha(\theta)$; Equation (\ref{eqn.bp_proof5}) uses monotone convergence and Fubini's theorem to swap the summation and integral, which simplifies to Equation (\ref{eqn.bp_proof6}) as a result of the exponential power series.

For the final step, we invert the product integral back into an exponential function of an integral; for a diffuse measure $\kappa$, we use the product integral equality \citep{Volterra:1959}
$$\textstyle\prod_{d\theta\in A}(1+\kappa(d\theta)) = \exp\textstyle\int_{d\theta\in A}\ln(1+\kappa(d\theta)) = \exp\textstyle\int_{d\theta\in A}\kappa(d\theta).$$ 
We note that the result of Equation (\ref{eqn.bp_proof6}) satisfies this condition since the right-most term, which corresponds to $\kappa(d\theta)$, is a finite number multiplied by an infinitesimal measure.\footnote{To give a sense of how this equality arises, we note that $\ln(1+\kappa(d\theta)) = \sum_{n=1}^{\infty} \frac{(-1)^{n+1}}{n!}\kappa(d\theta)^n$ since $|\kappa(d\theta)| < 1$, and all terms with $n > 1$ integrate to zero due to a lack of atoms, leaving only $\kappa(d\theta)$.} In light of Equation (\ref{eqn.bp_proof6}), the above product integral equality leads to
\begin{equation}
\E\,\e^{tH(A)} = \exp\left\lbrace-\int_{d\theta\in A}\int_0^1(1-\e^{t\pi}) \alpha(\theta)\pi^{-1}(1-\pi)^{\alpha(\theta)-1}d\pi \mu(d\theta)\right\rbrace\label{eqn.bp_proof7}.
\end{equation}
This is the Laplace functional of the integral form of the beta process given in Definition \ref{def.bp1} using Poisson random measures. By uniqueness of the Laplace functional, the $H$ given in both definitions share the same law, and so they are equivalent.
\end{proof}

\subsection*{A3. A second proof of Theorem \ref{thm.truncations} using the Poisson process}
 Let $U\in\{0,1\}^M$. By the marking theorem for Poisson processes, the set $\{(\theta,\pi,U)\}$ constructed from groups $R+1$ and higher is a Poisson process on $\rm{\Theta}\times [0,1]\times\{0,1\}^M$ with mean measure $\nu_R^+(d\theta,d\pi)Q(\pi,U)$ and a corresponding counting measure $N_R^+(d\theta,d\pi,U)$, where $Q(\pi,\cdot)$ is a transition probability measure on the space $\{0,1\}^M$. Let $A = \{0,1\}^M \backslash \boldsymbol{0}$, where $\boldsymbol{0}$ is the zero vector. Then $Q(\pi,A)$ is the probability of this set with respect to a Bernoulli process with parameter $\pi$, and therefore $Q(\pi,A) = 1 - (1-\pi)^M$. The probability $\mathbb{P}(E)$ equals $1 - \mathbb{P}(E^c)$, which is equal to $1 - \mathbb{P}(N_R^+(\rm{\Theta},[0,1],A)=0)$. The theorem follows since $N_R^+(\rm{\Theta},[0,1],A)$ is a Poisson-distributed random variable with parameter $\int_{(0,1]}\nu_R^+(\rm{\Theta},d\pi)Q(\pi,A)$.$\hfill\square$

\end{document}